\documentclass[11pt]{amsart}
\usepackage{amsmath}
\usepackage{}
\usepackage{amsfonts}
\usepackage{amsfonts,latexsym,rawfonts,amsmath,amssymb,amsthm}
\usepackage[plainpages=false]{hyperref}

\usepackage{graphicx}
\makeatletter
\renewcommand\@biblabel[1]{}
\makeatother

\numberwithin{equation}{section}

\newcommand{\beq}{\begin{equation}}
\newcommand{\eeq}{\end{equation}}
\newcommand{\beqs}{\begin{eqnarray*}}
\newcommand{\eeqs}{\end{eqnarray*}}
\newcommand{\beqn}{\begin{eqnarray}}
\newcommand{\eeqn}{\end{eqnarray}}
\newcommand{\beqa}{\begin{array}}
\newcommand{\eeqa}{\end{array}}

\def\lra{\longrightarrow}

\def\bc{\begin{center}}
\def\ec{\end{center}}

\def\begeq{\begin{equation}}
\def\endeq{\end{equation}}
\def\and{\quad{\rm and}\quad}

\let\lra=\longrightarrow

\def\mapright\#1{\,\smash{\mathop{\lra}\limits^{\#1}}\,}

\newtheorem{prop}{Proposition}[section]
\newtheorem{theo}[prop]{Theorem}
\newtheorem{lem}[prop]{Lemma}

\newtheorem{cor}[prop]{Corollary}
\newtheorem{rem}[prop]{Remark}

\newtheorem{defi}[prop]{Definition}

\begin{document}

\date{}
\author {Yuxing Deng }
\author { Xiaohua $\text{Zhu}^*$}

\thanks {* Partially supported by the NSFC Grants 11271022 and 11331001}

\address{ Yuxing Deng\\School of Mathematical Sciences, Beijing Normal University,
Beijing, 100875, China\\
dengyuxing@mail.bnu.edu.cn}

\address{ Xiaohua Zhu\\School of Mathematical Sciences and BICMR, Peking University,
Beijing, 100871, China\\
xhzhu@math.pku.edu.cn}

\title{ 3d steady Gradient Ricci Solitons with linear curvature decay}
\maketitle

\section*{\ }

\begin{abstract} In this note,  we   prove that  a 3-dimensional steady Ricci soliton is rotationally symmetric if  its scalar curvature $R(x)$   satisfies
 $$\frac{C_0^{-1}}{\rho(x)}\le R(x)\le \frac{C_0}{\rho(x)}$$
   for some constant $C_0>0$,   where $\rho(x)$ denotes the distance  from a fixed point $x_0$.
 Our result doesn't assume that the soliton is $\kappa$-noncollapsed.
\end{abstract}

\section{Introduction}

 In his celebrated paper \cite{Pe1},   Perelman conjectured that \textit{all 3-dimensional $\kappa$-noncollapsed steady (gradient) Ricci solitons must be rotationally symmetric}.
The conjecture is  solved by Brendle in 2012 \cite{Br1}.   For a general  dimension $n\ge 3$,  under an extra condition that   the soliton is  asymptotically   cylindrical,    Brendle  also  proves that   any $\kappa$-noncollapsed  steady Ricci soliton  with positive sectional curvature must be rotationally symmetric \cite{Br2}.   In general, it is still open  \textit{whether an $n$-dimensional $\kappa$-noncollapsed steady Ricci soliton with positive curvature operator is rotationally symmetric for $n\ge 4$}.  For $\kappa$-noncollapsed steady  K\"{a}hler-Ricci  solitons with nonnegative bisectional curvature, the authors have recently proved   that they must be flat \cite{DZ3}, \cite{DZ4}.

Recall  from \cite{Br2},

 \begin{defi}\label{brendle} An $n$-dimensional steady  Ricci soliton $(M,g,f)$  is called  asymptotically  cylindrical   if   the following
holds:

(i) Scalar curvature $R(x)$  of $g$ satisfies
$$\frac{C_0^{-1}}{\rho(x)} \le  R(x) \le  \frac{C_0}{\rho(x)},~\forall~\rho(x)\ge r_0, $$
where $C_0>0$  is a  constant and $\rho(x)$ denotes the distance of $x$  from a fixed  point $x_0$.

(ii) Let $p_m$ be an arbitrary sequence of marked points going to infinity.
Consider  rescaled metrics
$ g_m(t) = r_m^{-1} \phi^*_{r_m t} g,$ where
$r_m R(p_m) = \frac{n-1}{2} + o(1)$ and $ \phi_{ t}$ is a one-parameter subgroup generated by $X=-\nabla f$.   As  $m \to\infty,$
 flows $(M,  g_m(t), p_m)$
converge in the Cheeger-Gromov sense to a family of shrinking cylinders
$(   \mathbb R \times \mathbb S^{n-1}(1), \widetilde g(t)), t \in  (0, 1).$  The metric $\widetilde g(t)$  is given by
\begin{align}\label{n-2-behavior}  \widetilde g(t) =   dr^2+ (n - 2)(2 -2t) g_{\mathbb S^{n-1}(1)},
\end{align}
 where $\mathbb S^{n-1}(1)$ is the unit sphere of euclidean space.
\end{defi}

In this note, we   discuss   3-dimensional steady (gradient) Ricci solitons without assuming    the  $\kappa$-noncollapsed conditon.\footnote{It is proved by Chen  that any  3-dimensional ancient solution has nonnegative sectional curvature \cite{Ch}.}    We prove

 \begin{theo}\label{main-theorem}
 Let $(M,g,f)$ be a  3-dimensional steady Ricci soliton. Then, it is rotationally symmetric if the scalar curvature $R(x)$ of $(M,g,f)$ satisfies
 \begin{align}\label{curvature decay}
 \frac{C_0^{-1}}{\rho(x)}\le R(x)\le \frac{C_0}{\rho(x)},
 \end{align}
 for some constant $C >0$,   where $\rho(x)$ denotes the distance  from a fixed point $x_0$.
 \end{theo}

Under the condition  (\ref{curvature decay}),  we need to check the  property (ii) in Definition \ref{brendle} to prove Theorem \ref{main-theorem}.   Actually, we show that  for any sequence $p_{i}\rightarrow\infty$, there exists a subsequence $p_{i_{k}}\rightarrow\infty$ such that
\begin{align*}
(M,g_{p_{i_{k}}}(t),p_{i_{k}})\rightarrow(\mathbb{R}\times\mathbb{S}^{2},\widetilde{g}(t),p_{\infty}),~for~t\in(-\infty,1),
\end{align*}
where $g_{p_{i_{k}}}(t)=R(p_{i_{k}})g(R^{-1}(p_{i_{k}})t)$ and $(\mathbb{R}\times\mathbb{S}^{2},\widetilde{g}(t))$ is a shrinking cylinders flow, i.e.
$$\widetilde{g}(t)=dr^2+(2-2t)g_{\mathbb{S}^{2}}.$$

As in \cite{DZ3},  we  study the geometry of neighborhood $M_{p,k}=\{x\in M| ~ f(p)-\frac{k}{\sqrt{R(p)}}\le f(x)\le f(p)+\frac{k}{\sqrt{R(p)}}\}$  around  level set $\Sigma_r=\{f(x)=f(p)=r\}$ for any $p\in M$.  We are able to give a uniform injective radius estimate for $(M,R_{p_i}g)$ at each sequence of $p_i$. Then  we can still get a limit flow for  rescaled  flows $(M,g_{p_{i}}(t))$,   which  will  split off a  line.  By using a  classification result of Daskalopoulos-Hamilton-Sesum  for  ancient flows on a compact surface  \cite{DHS}, we  finish the proof of Theorem \ref{main-theorem}.

  We remark that the curvature condition  in Theorem \ref{main-theorem} cannot be removed,  since there does exist a 3-dimensional non-flat  steady Ricci soliton with exponential curvature decay.  For example, $(\mathbb{R}^2\times\mathbb{S}^{1},g_{cigar}+ds^2)$, where $(\mathbb{R}^2,g_{cigar})$ is a  cigar soliton.
Also, Theorem \ref{main-theorem} is not true for dimension $n\ge 4$ by Cao's examples of steady  K\"{a}hler-Ricci  solitons with positive  sectional curvature \cite{Ca2}.

At last, we remark that it is still  open   \textit{wether there exists a 3-dimensional collapsed steady Ricci soliton with positive  curvature}.  Hamilton has conjectured that \textit{there should exist a family of collapsed 3-dimensional complete gradient steady Ricci solitons with positive curvature and $S^1$-symmetry} (cf. \cite{C-He}). Our result shows that 
the curvature of  Hamilton's examples could not  have a linear decay.

\noindent {\bf Acknowledgements.} The work was done partially when the second author was   visiting at the Mathematical Sciences Research Institute at Berkeley during the spring 2016 semester. 
 The author  would like to thank her hospitality and the financial supports,    National Sciences Foundation under Grant No. DMS-1440140, and  Simons Foundation.

\section{Positivity of Ricci curvature}
$(M,g,f)$ is called a  gradient steady Ricci soliton if  a Riemannian metric $g$ on $M$ satisfies
\begin{align}
{\rm Ric}(g)=\nabla^2 f,
\end{align}
for some smooth function $f$.   We first show  the positivity of Ricci curvature of $(M,g,f)$
  under  (\ref{curvature decay})  of Theorem \ref{main-theorem}.

\begin{lem}\label{positive curvature}
Under   (\ref{curvature decay}),     $(M,g,f)$ has positive sectional curvature.
\end{lem}
\begin{proof}

 We need to show that $(M,g)$ has positive Ricci curvature.
   On the contrary,  $(M,g)$ locally splits off a flat piece of line by Shi's splitting theorem \cite{S2}.
    Then, the universal cover $(\widetilde{M},\widetilde{g})$ of $(M,g)$ is isometric to a product Riemannian manifold of a real line and the cigar soliton. Namely,  $(\widetilde{M},\widetilde{g})=(\mathbb{R}^2\times\mathbb{R},g_{cigar}+ds^2)$. Let  $\pi:\widetilde{M}\to M$ be  a universal covering. We fix $x_0\in M$ and $\widetilde{x}_0\in \widetilde{M}$ such that $\pi(\widetilde{x}_0)=x_0$. For any $x\in M$ and $\widetilde{x}\in \widetilde{M}$ such that $\pi(\widetilde{x})=x$, one sees
 \begin{align}\label{positive curvature-1}
 \rho(x,x_0)\le \widetilde{\rho}(\widetilde{x},\widetilde{x}_0),
  \end{align}
  where $\rho$ and $\widetilde{\rho}$ are the distance functions w.r.t $g$ and $\widetilde{g}$ respectively. Let $\{\widetilde{x}_i\}_{i\ge1}$ be a sequence of points  so that $\widetilde{x}_i=(p_i,0)\in \mathbb{R}^2\times\mathbb{R}$ tend to infinity. Then, one may check that
 \begin{align}\label{positive curvature-2}
 \widetilde{R}(\widetilde{x}_i)\rho(\widetilde{x}_i,\widetilde{x}_0)\to0,~{\rm as}~i\to\infty.
 \end{align}
Since $R(x_i)=\widetilde{R}(\widetilde{x}_i)\to0$, where $x_i=\pi(\widetilde{x}_i)$,   we  see $d(x_i,x_0)\to \infty$ by (\ref{curvature decay}).
Again by (\ref{curvature decay}) and (\ref{positive curvature-1}), we get
\begin{align}
C_1\le R(x_i)d(x_i,x_0)\le \widetilde{R}(\widetilde{x}_i)d(\widetilde{x}_i,\widetilde{x}_0).
\end{align}
This is a contradiction to (\ref{positive curvature-2}). Hence,  the lemma is  proved.
\end{proof}

\begin{cor}\label{equilibrium-existence} $(M,g,f)$  in Theorem \ref{main-theorem} has a unique equilibrium point $o$, i.e., $ \nabla f(o)=0$. As a consequence,  $\Sigma_r=\{f(x)=r\}$ is diffeomorphic to $\mathbb{S}^{2}$, for any $r>f(o)$.
\end{cor}

\begin{proof}

 Note that
\begin{align}\label{identity}
|\nabla f|^2+R=A.
\end{align}
By taking covariant derivatives on both sides of  (\ref{identity}), it follows
\begin{align}
2{\rm Ric}(\nabla f,\nabla f)=-\langle\nabla R,\nabla f\rangle.
\end{align}
On the the hand,  by (\ref{curvature decay}),  there exists a point $o$ such that
$$\sup_M R(x)=R(o)=R_{{\rm max}}.$$
In particular, $\nabla R(o)=0.$
Thus
$${\rm Ric}(\nabla f,\nabla f)=0.$$
By Lemma \ref{positive curvature},  $\nabla f(o)=0$.  The uniqueness also follows from the positivity of Ricci curvature.

 By the Morse theorem,  $\Sigma_r=\{f(x)=r>f(o)\}$ is diffeomorphic to $\mathbb{S}^{2}$ (cf. \cite{DZ3}, Lemma 2.1).

\end{proof}

\section{Geometry of $M_{p,k}$}

For any $p\in M$ and number $k>0$, we set
$$M_{p,k}=\{x\in M| ~ f(p)-\frac{k}{\sqrt{R(p)}}\le f(x)\le f(p)+\frac{k}{\sqrt{R(p)}}\}.$$
Let $g_{p}=R(p)g$ be a rescaled metric and denote $B(p,r;  g_{p})$ a $r$-geodesic ball centered at $p$ with respect to $g_{p}$.
Then by Corollary \ref{equilibrium-existence}, we have (cf. \cite{DZ3}, Lemma 3.3)

\begin{lem}\label{set-mr-contain-1}
Under   (\ref{curvature decay}),  for   any $p\in M$ and number $k>0$ with $f(p)-\frac{k}{\sqrt{R(p)}}>f(o)$,   it holds
\begin{align}\label{set-mr-contain}
B(p,\frac{k}{\sqrt{R_{max}}};  g_{p})\subset M_{p,k}.
\end{align}

\end{lem}

By Lemma \ref{set-mr-contain-1}, we prove

\begin{lem}\label{lem-pointwise curvature estimate}
Under   (\ref{curvature decay}), there exists a constant $C$ such that
\begin{align}\label{pointwise curvature estimate}
\frac{|\Delta R|(p)}{R^2(p)}\le C,~\forall~p\in ~M.
\end{align}
\end{lem}

\begin{proof}
Fix any $p\in M$ with $f(p)\ge r_0>>1$. Then
\begin{align}\label{inequality-set}
|f(x)-f(p)|\le \frac{1}{\sqrt{R(p)}},~ \forall ~x\in M_{p,1}.
\end{align}
It is known by  \cite{CaCh},
\begin{align}\label{inequality-cao chen}
c_1\rho(x)\le f(x)\le c_2 \rho(x), ~\forall~\rho(x)\ge r_0.
\end{align}
Thus by (\ref{curvature decay}), (\ref{inequality-set}) and (\ref{inequality-cao chen}), we get
\begin{align}
c_2\rho(x)\ge f(p)-\frac{1}{\sqrt{R(p)}}\ge c_1\rho(p)-\sqrt{C_0\rho(p)}.\notag
\end{align}
It follows
\begin{align}\label{upper bound}
\frac{R(x)}{R(p)}\le  C_0^2\frac{\rho(p)}{\rho(x)}\le \frac{2c_2C_0^2}{c_1},~\forall x\in M_{p,1}.
\end{align}
On the other hand, by (\ref{set-mr-contain}), we have
\begin{align}
B(p,\frac{1}{\sqrt{R_{\max}}};g_p)\subseteq M_{p,1}.\notag
\end{align}
Hence
\begin{align}\label{curvature bound-1}
R(x)\le C^{\prime}R(p),~\forall ~x\in B(p,\frac{1}{\sqrt{R_{\max}}};g_p).
\end{align}

Let $\phi_t$ be generated by $-\nabla f$. Then $g(t)=\phi_t^{\ast}g$ satisfies the Ricci flow,
\begin{align}\label{Ricci flow equation}
         \frac{\partial g(t)}{\partial t} &= -2{\rm Ric}(g(t)).
                  \end{align}
Also rescaled flow  $g_p(t)=R(p) g(R^{-1}(p) t)$  satisfies (\ref{Ricci flow equation}). Since the Ricci curvature is positive,
\begin{align}
B(p,\frac{1}{\sqrt{R_{\max}}};g_p(-t))\subseteq B(p,\frac{1}{\sqrt{R_{\max}}};g_p(0)),~t\in~[-1,0].\notag
\end{align}
Combining with  (\ref{curvature bound-1}), we get
\begin{align}
R_{g_p(t)}(x)\le C^{\prime},~\forall~x\in B(p,\frac{1}{\sqrt{R_{\max}}};g_p(0)),~t\in [-1,0].
\end{align}
Thus,  by Shi's higher order estimates, we obtain
\begin{align}
|\Delta_{g_p(t)} R_{g_p(t)}|(x)\le C_{1}^{\prime},~\forall~x\in B(p,\frac{1}{ 2\sqrt{R_{\max}}};g_p(-1)),~t\in  [-\frac{1}{2},0].\notag
\end{align}
It follows
\begin{align}
|\Delta R|(x)\le C_{1}^{\prime}R^2(p),~\forall~x\in B(p,\frac{1}{2\sqrt{R_{\max}}};g_p(-1)).\notag
\end{align}
In particular, we have
\begin{align}
|\Delta R|(p)\le C_{1}^{\prime}R^2(p),~{\rm as}~\rho(p)\ge r_0.\notag
\end{align}
The lemma is  proved.
\end{proof}

\begin{rem}\label{rem-pointwise curvature estimate}
Under (\ref{curvature decay}),   by the same argument as in the proof of  Lemma \ref{lem-pointwise curvature estimate}, for each $k\in\mathbb{N}$, there exists a constant $C(k)$ such that
\begin{align}
\frac{|\nabla^k R|(p)}{R^{\frac{k+2}{2}}(p)}\le C(k),~\forall~p\in ~M.
\end{align}
\end{rem}

Next, we want to show that  $M_{p,k}$ is bounded by  a finite  ball $B(p, Ck ; g_p)$, where $C$ is a uniform constant. We need
to use the Gauss formula,
\begin{align*}
R(X,Y,Z,W)=\overline{R}(X,Y,Z,W)+\langle B(X,Z),B(Y,W)\rangle-\langle B(X,W),B(Y,Z)\rangle,
\end{align*}
where $X,Y,Z,W\in T\Sigma_{r}$ and $B(X,Y)=(\nabla_{X}Y)^{\bot}$. Note that
\begin{align*}
B(X,Y)=&\langle \nabla_{X}Y,\nabla f\rangle\cdot\frac{\nabla f}{|\nabla f|^{2}}\\
=&[\nabla_{X}\langle Y,\nabla f\rangle-\langle Y,\nabla_{X}\nabla f\rangle]\cdot\frac{\nabla f}{|\nabla f|^{2}}\\
=&-{\rm Ric}(X,Y)\cdot\frac{\nabla f}{|\nabla f|^{2}}.
\end{align*}
 We choose a normal basis $\{e_1,e_2\}$  on $(\Sigma_r,\bar g)$  with the induced metric $\bar g$. Then  $\{e_1,e_2,\frac{\nabla f}{|\nabla f|}\}$ spans a normal basis of $(M,g)$. Thus
\begin{align}
R_{11}=&\overline{R}_{11}+R(\frac{\nabla f}{|\nabla f|},e_{1},e_{1},\frac{\nabla f}{|\nabla f|})-\frac{R_{11}R_{22}-R_{12}R_{21}}{|\nabla f|^{2}},\notag\\
R_{22}=&\overline{R}_{22}+R(\frac{\nabla f}{|\nabla f|},e_{2},e_{2},\frac{\nabla f}{|\nabla f|})-\frac{R_{11}R_{22}-R_{12}R_{21}}{|\nabla f|^{2}}.\notag
\end{align}
Since $(\Sigma_r, \bar g)$ is a surface,  $K=\overline{R}_{11}=\overline{R}_{22}$. Hence, we get

\begin{lem}\label{gauss-curvature}
The Gauss curvature of $(\Sigma_r,  \bar g)$ is given by
\begin{align}
K=\frac{R}{2}-\frac{{\rm Ric}(\nabla f,\nabla f)}{|\nabla f|^2}+\frac{R_{11}R_{22}-R_{12}R_{21}}{|\nabla f|^{2}}.
\end{align}
\end{lem}

\begin{lem}\label{set-mr-contain-2}
Under   (\ref{curvature decay}),   there exists  a uniform $B>0$ such that the following is true:
  for any $k\in\mathbb{N}$, there exists $\bar r_0=\bar r_0(k)$ such that
\begin{align}
M_{p,k}\subset B(p,2\pi\sqrt{B}+\frac{2k}{\sqrt{R_{\max}}} ; g_p), ~\forall ~ \rho(p)\ge\bar  r_0.
\end{align}
\end{lem}

\begin{proof}
By (\ref{curvature decay}) and   (\ref{inequality-cao chen}),  we have
\begin{align}\label{bound of gradient f}
\frac{R_{\max}}{2}\le |\nabla f|^2(x)\le R_{\max},~\forall~x\in M_{p,k},
\end{align}
as long as $\rho(p)\ge r_0>>1.$
Then by Lemma \ref{gauss-curvature} and  Lemma \ref{lem-pointwise curvature estimate},  we get
\begin{align}
|K-\frac{R}{2}|=&|-\frac{{\rm Ric}(\nabla f,\nabla f)}{|\nabla f|^2}+\frac{R_{11}R_{22}-R_{12}R_{21}}{|\nabla f|^{2}}|\notag\\
\le&\frac{|\langle\nabla R,\nabla f\rangle|}{2|\nabla f|^2}+\frac{R^2}{|\nabla f|^{2}}\notag\\
\le&\frac{|\Delta R+2|{\rm Ric}|^2|}{2|\nabla f|^2}+\frac{R^2}{|\nabla f|^{2}}\notag\\
\le&\frac{(C+4)R^2}{R_{\max}}.\notag
\end{align}
It follows
\begin{align}\label{gauss curvature bound by scalar}
\frac{R(x)}{4}\le K(x)\le\frac{3R(x)}{4},~\forall ~x\in M_{p,k},  ~\rho(p)\ge r_0.
\end{align}

On the other hand, by (\ref{curvature decay}),
  (\ref{inequality-set}) and (\ref{inequality-cao chen}),  we see
\begin{align}\label{distance estimate}
 c_2^{-1}\Big(c_1\rho(p)-k\sqrt{\rho(p)C_0}\Big)&\le \rho(x) \notag\\
 &\le c_1^{-1}(c_2\rho (p)+k\sqrt{\rho(p)C_0}),~\forall ~x\in M_{p,k},
\end{align}
as long as  $\rho(p)\ge r_0.$
Then  similar to (\ref{upper bound}),   there exists a $\bar r_0\ge r_0$ such that
\begin{align}\label{lower-bound}R(x)\ge \frac{c_1}{2c_2C_0^2} R(p), ~\forall ~x\in M_{p,k}.
\end{align}
Thus  by (\ref{gauss curvature bound by scalar}), we get
\begin{align*}
\overline{R}_{ij}\geq  B^{-1} R(p)\overline{g}_{ij}, ~\forall ~x\in \Sigma_{f(p)}, ~\rho(p)\ge \bar r_0,
\end{align*}
    where $B>0$ is a uniform constant.   By the Myer's theorem,  the diameter of $\Sigma_{f(p)}$ is bounded by
\begin{align}
{\rm diam}(\Sigma_{f(p)},g)\le {\rm diam}(\Sigma_{f(p)},\overline{g}_{f(p)})\leq 2\pi\sqrt{\frac{B}{ R(p)}}.\notag
\end{align}
As a consequence,
\begin{align}\label{mr-set}
\Sigma_{f(p)}\subset B(p,2\pi \sqrt{B};  R(p)g).
\end{align}

For any $q\in M_{p,k}$, there exists  $q^{\prime}\in \Sigma_{f(p)}$  such that $\phi_{s}(q)=q^{\prime}$ for some $s\in \mathbb{R}$.
Then by (\ref{mr-set}) and  (\ref{bound of gradient f}), we have

\begin{align*}
d(q,p)\leq& d(q^{\prime},p)  + d(q,q^{\prime})\\
\leq& {\rm diam}(\Sigma_{f(p)},g)+\mathcal{L}(\phi_{\tau}|_{[0,s]})\\
\leq& 2\pi\sqrt{\frac{B}{R(p)}}+|\int_{0}^{s}|\frac{d\phi_{\tau}(q)}{d\tau}|d\tau|\\
=& 2\pi\sqrt{\frac{B}{ R(p)}}+\int_{0}^{s}|\nabla f(\phi_{\tau}(q))|d\tau\\
\le& 2\pi\sqrt{\frac{B}{R(p)}}+\int_{0}^{s}|\nabla f(\phi_{\tau}(q))|^{2}\cdot \frac{2}{\sqrt{R_{\max}}}d\tau\\
=& 2\pi\sqrt{\frac{B}{ R(p)}}+|\int_{0}^{s}\frac{d(f(\phi_{\tau}(q)))}{d\tau}\cdot \frac{2}{\sqrt{R_{\max}}}d\tau|\\
\leq& 2\pi\sqrt{\frac{B}{R(p)}}+|f(q)-f(p)|\cdot \frac{2}{\sqrt{R_{\max}}}\\
\leq& \Big(2\pi\sqrt{B}+\frac{2k}{\sqrt{R_{\max}}}\Big)\cdot \frac{1}{\sqrt{R(p)}}.
\end{align*}
Thus
\begin{align*}
M_{p,k}\subset B(p,2\pi\sqrt{B}+\frac{2k}{\sqrt{R_{\max}}} ; R(p)g).
\end{align*}
 The lemma is proved.

\end{proof}

By Lemma \ref{set-mr-contain-2}, we  get  the following volume estimate of  $B(p,s;g_p)$.

\begin{prop}\label{volume estimate}
Under    (\ref{curvature decay}) of Theorem  \ref{main-theorem}, there exists $s_0$ and ${c}>0$ such that
\begin{align}\label{volume-s}
{\rm Vol}B(p,s;g_p)\ge  c s^3,~\forall~s\le s_0~{\rm and}~\rho(p)\ge r_0>>1.
\end{align}
Moreover, the injective radius of $(M,g_p)$ at $p$ has a uniform lower bound $\delta>0$, i.e.,
\begin{align}\label{injective}
{\rm inj}(p,g_p)\ge \delta,~\forall~\rho(p)\ge r_0.
\end{align}
\end{prop}
\begin{proof}

By Lemma  \ref{set-mr-contain-2},  we have
\begin{align}
M_{p,1}\subset B(p,2\pi\sqrt{B}+\frac{2}{\sqrt{R_{\max}}} ; g_p).\notag
\end{align}
In the following,
we give an estimate of  ${\rm Vol}(\Sigma_l, \bar g)$ for  any $l$  with $f(p)-\frac{1}{\sqrt{R(p)}}\le l\le f(p)+\frac{1}{\sqrt{R(p)}}$.

   By (\ref{upper bound}) and (\ref{lower-bound}), we see
\begin{align}
{C}_1^{-1}\le\frac{R(x)}{R(p)}\le {C}_1, ~\forall~\rho(p)\ge r_0~{\rm and}~x\in M_{p,1}.\notag
\end{align}
By (\ref{gauss curvature bound by scalar}), it follows that the Gauss curvature $K_l$ of $(\Sigma_l, g_p|_{\Sigma_l})$ satisfies
\begin{align}
\frac{1}{4C_1}\le K_l\le \frac{3C_1}{4}.\notag
\end{align}
Thus
\begin{align}
{\rm Vol}(\Sigma_l, \bar g)=\frac{1}{R(p)}{\rm Vol}(\Sigma_l,g_p|_{\Sigma_l})\ge \frac{64\pi {C}_1}{R(p)}.\notag
\end{align}
By  the Co-Area formula, we get
\begin{align}
{\rm Vol}(M_{p,1},g)=&\int_{f(p)-\frac{1}{\sqrt{R(p)}}}^{f(p)+\frac{1}{\sqrt{R(p)}}} \frac{{\rm Vol}(\Sigma_l, \bar g )}{|\nabla f|} dl  \notag\\
\ge&  128\pi  {C}_1 R_{max}^{-\frac{1}{2}} R^{-\frac{3}{2}}(p). \notag
\end{align}
Hence
\begin{align}\label{volume-ball}
{\rm Vol}(B(p,2\pi\sqrt{B}+\frac{2}{\sqrt{R_{\max}}} ; g_p))\ge {\rm Vol}(M_{p,1},g_p)\ge 128\pi {C}_1 R_{max}^{-\frac{1}{2}}.
\end{align}

By the volume comparison theorem, we derive from (\ref{volume-ball}),
\begin{align}
\frac{{\rm Vol}(B(p,s; g_p))}{s^3}\ge& \frac{{\rm Vol}(B(p,2\pi\sqrt{B}+\frac{2}{\sqrt{R_{\max}}} ; g_p))}{(2\pi\sqrt{B}+\frac{2}{\sqrt{R_{\max}}})^3}\notag\\
\ge&\frac{128\pi {C}_1R_{max}^{-\frac{1}{2}}  }{(2\pi\sqrt{B}+\frac{2}{\sqrt{R_{\max}}})^3},\notag
\end{align}
for any $s\le2\pi\sqrt{B}+\frac{2}{\sqrt{R_{\max}}}$.  This proves (\ref{volume-s}).  By   (\ref{volume-s}),  we can apply  a  result of Cheeger-Gromov-Taylor  for Riemannian manifolds with bounded curvature to get the injective radius estimate (\ref{injective})  immediately  \cite{CGT}.

\end{proof}

\section{Proof of Theorem \ref{main-theorem}}

First  we prove the following  convergence of rescaled flows.

\begin{lem}\label{lem-for the main theorem-1}
Under    (\ref{curvature decay}), let $p_i\to\infty$.   Then  by taking a subsequence of $p_i$ if necessary,   we have
\begin{align*}
(M,g_{p_{i}}(t),p_{i})\rightarrow(\mathbb{R}\times N,\widetilde{g}(t); p_{\infty}),~for~t\in(-\infty,0],
\end{align*}
where $g_{p_i}(t)=R(p_{i})g(R^{-1}(p_{i})t)$, $\widetilde{g}(t)=ds\otimes ds+g_{ N}(t)$ and $( N, g_{N}(t))$ is an ancient solution of Ricci flow on  $N$.
\end{lem}

\begin{proof}
 For a fixed $\overline{r}$,   as in (\ref{upper bound}),   it is easy to see  that there exists  a uniform $C_1$  independent of $\overline{r}$  such that
\begin{align}\label{scalar-estimate-2}
R(x)\le C_1R(p_i),~\forall~x\in M_{p_i,\overline{r}\sqrt{R_{\max}}}
\end{align}
as long as $i$ is large enough.  By Lemma \ref{set-mr-contain-1}, it follows
\begin{align}
R_{g_{p_i}}(x)\le C_1,~\forall~x\in B(p_i,\overline{r};g_{p_i}),\notag
\end{align}
where $g_{p_i}=g_{p_i}(0)$.   Since the scalar curvature is increasing along the flow,  we get
\begin{align}
|{\rm Rm}_{g_{p_i}(t)}(x)|_{g_{p_i}(t)}&\le 3R_{g_{p_i}(t)}(x)\notag\\
&\le 3R_{g_{p_i}}(x)\le 3C_1,~\forall~x\in B(p_i,\overline{r};g_{p_i}),~t\in(-\infty,0].\notag
\end{align}
Thus  together with   the injective radius estimate in Proposition  \ref{volume estimate},  we  can apply the Hamilton compactness theorem to see that  $g_{p_i}(t)$ converges subsequently  to a limit flow $(\widetilde M, \tilde g(t); p_\infty)$ on  $t\in(-\infty,0]$ \cite{H1}.
Moreover,  the limit flow  has uniformly bounded curvature.  It remains to prove the splitting property.

By Remark \ref{rem-pointwise curvature estimate},   we have
$$|{\rm  Ric}|(x)\le CR(x),~\forall~x\in B(p_{i},\overline{r} ;  {g_{p_i}}).$$
It follows from (\ref{scalar-estimate-2}),
\begin{align}
|{\rm  Ric}|(x)\le CR(p_i),~\forall~x\in B(p_{i},\overline{r} ;  {g_{p_i}}).\notag
\end{align}
Let $X_{(i)}=R(p_{i})^{-\frac{1}{2}}\nabla f$.
Then
\begin{align}
\sup_{ B(p_{i},r_{0} ;  {g_{p_i}})}| \nabla_{(g_{p_i})}X_{(i)}|_{g_{p_i}}&= \sup_{ B(p_{i},r_{0} ;  {g_{p_i}})}\frac{|{\rm Ric}|}{\sqrt{R(p_{i})}}\notag\\
&\le C\sqrt{R(p_{i})} \to 0.\notag
\end{align}
On the other hand,  by Remark \ref{rem-pointwise curvature estimate},  we also have
$$\sup_{ B(p_{i},r_{0} ;  {g_{p_i}})}| \nabla^{m}_{(g_{p_i})}X_{(i)}|_{g_{p_i}}\leq C(n)\sup_{ B(p_{i},r_{0} ;  {g_{p_i}})}| \nabla^{m-1}_{(g_{p_i})}{\rm Ric}({g_{p_i})}|_{g_{p_i}}\le C_1.$$
Thus $X_{(i)}$ converges  subsequently  to a parallel  vector field $X_{(\infty)}$ on $(\widetilde M, \tilde g(0))$.
 Moreover,
 \begin{align}
 |X_{(i)}|_{g_{p_i}}( x)=|\nabla f|(p_{i})
=\sqrt{R_{\rm max}}+o(1)>0, ~\forall~ x\in B(p_{i},r_{0} ;  {g_{i}}),\notag
 \end{align}
as long as $f(p_i)$ is large enough.  This implies that $X_{(\infty)}$ is non-trivial.
 Hence,  $(\widetilde M,\widetilde{g}(t))$ locally splits off a piece of  line along $X_{(\infty)}$. It remains to show that  $X_{(\infty)}$
 generates a line through $p_\infty$.
 
 By Lemma \ref{set-mr-contain-2},
\begin{align}
M_{p_i,k}\subset B(p_i,2\pi\sqrt{B}+\frac{2k}{\sqrt{R_{\max}}} ; g_{p_i}(0)), ~\forall ~ p_i\to\infty.\notag
\end{align}
Let $\gamma_{i,k}(s)$, $s\in(-D_{i,k},E_{i,k})$ be an integral curve generated by $X_{(i)}$ through $p_i$, which  restricted in $M_{p,k}$. Then $\gamma_{i,k}(s)$ converges to a  geodesic $\gamma_\infty(s)$ generated by $X_{(\infty)}$ through $p_\infty$,  which restricted in
$B(p_\infty,2\pi\sqrt{B}+\frac{2k}{\sqrt{R_{\max}}};\widetilde{g}(0))$.  If let  $L_{i,k}$ be   lengths  of $\gamma_{i,k}(s)$
and $L_{\infty,k}$  length  of    $\gamma_\infty(s)$,
\begin{align}
L_{i,k}=&\int_{-D_{i,k}}^{E_{i,k}}|\nabla f|_{g_{p_i}(0)} ds=\int_{f(p_i)-\frac{k}{\sqrt{R(p_i)}}}^{f(p_i)+\frac{k}{\sqrt{R(p_i)}}}  \sqrt{R(p_i)} \|\nabla f\|_{g} \notag\\
&\ge  R_{max}^{-\frac{1}{2}} \int_{f(p_i)-\frac{k}{\sqrt{R(p_i)}}}^{f(p_i)+\frac{k}  {\sqrt{R(p_i)}}}  \sqrt{R(p_i)} \|\nabla f\|_{g}^2  ds\notag\\
&= R_{max}^{-\frac{1}{2}}      \int_{f(p_i)-\frac{k}{\sqrt{R(p_i)}}}^{f(p_i)+\frac{k}  {\sqrt{R(p_i)}}}  \sqrt{R(p_i)} df =2 R_{max}^{-\frac{1}{2}}k, \notag
\end{align}
and so, 
\begin{align}
L_{\infty,k}\ge \frac{1}{2}L_{i,k}\ge  R_{max}^{-\frac{1}{2}} k. \notag
\end{align}
 Thus  $X_{(\infty)}$ generates a line   $\gamma_\infty(s)$ through $p_\infty$ as $k\to \infty$. As a consequence,  $(\widetilde M,\widetilde{g}(0))$ splits off a  line 
 and so does the flow $(\widetilde M,\widetilde{g}(t); p_{\infty})$. The lemma is proved.

\end{proof}

Next we estimate the curvature of  $(N,g_{N}(t))$.

\begin{lem}\label{lem-for the main theorem-2}
Under  (\ref{curvature decay}),  there exists a constant $C$ independent of $t$ such that the scalar curvature $R_{N}(t)$ of $(N,g_{N}(t))$ satisfies
 \begin{align}\label{scalar curvature comparison}
 \frac{R_{N}(x,t)}{R_{N}(y,t)}\le C,~\forall~x,y\in N,~t\in(-\infty,0].
 \end{align}
\end{lem}

\begin{proof}
 Let $\widetilde{R}(x, t)$ be the scalar curvature of $(  \mathbb{R}\times N, \widetilde{g}(t))$. It suffices to prove the following is true:
 \begin{align}\label{curvature comparison}
 \frac{\widetilde{R}(x,t)}{\widetilde{R}(y,t)}\le C,~\forall~x,y\in\mathbb{R}\times N,~t\in(\infty,0],
 \end{align}
 for some constant $C$.   For any $x,y\in \mathbb{R}\times N$,  we choose  $\overline{r}>0$  such  that $x,y\in B(p_{\infty},\overline{r};\widetilde{g}(0))$. By the convergence of $g_{p_i}(t)$,  there are sequences  $\{x_i\}$ and $\{y_i\}$ in $ B(p_{i},\overline{r}; {g}_{p_i}(0))$  such that $x_i$ and $y_i$ converge to $x$ and $y$ in the Cheeger-Gromov sense, respectively.   By Lemma \ref{set-mr-contain-1}, we have
  \begin{align}
  x_i,y_i\subseteq B(p_{i},\overline{r}; {g}_{p_i}(0))\subseteq M_{p_i,\overline{r}\sqrt{R_{\max}}}.\notag
  \end{align}
  Thus
  \begin{align}\label{xi-yi}
  f(x_i)= (1+o(1))f(p_i)~{\rm and}~  f(y_i)=(1+o(1))f(p_i),  ~{\rm as}~p_i\to \infty.
  \end{align}

  On the other hand,  for a  fixed $t<0$,
 \begin{align}
\frac{f(\phi_{R^{-1}(p_i)t}(x_i))-f(x_i)}{|R^{-1}(p_i)t|}=\frac{\int_{R^{-1}(p_i)t}^0  |\nabla f|^2ds}{|R^{-1}(p_i)t|}\rightarrow R_{\max},~{\rm as}~p_i\to\infty\notag
 \end{align}
and
\begin{align} \frac{f(\phi_{R^{-1}(p_i)t}(y_i))-f(y_i)}{|R^{-1}(p_i)t|}=\frac{\int_{R^{-1}(p_i)t}^0  |\nabla f|^2ds}{|R^{-1}(p_i)t|}\rightarrow R_{\max},~{\rm as}~p_i\to\infty.\notag
 \end{align}
 By (\ref{xi-yi})  and the fact
 \begin{align}
 C_1\le R(x)f(x)\le C_2,~\forall~f(x)>>1, \notag
 \end{align}
 we get
 $$\frac{f(\phi_{R^{-1}(p_i)t}(x_i))}{f(\phi_{R^{-1}(p_i)t}(y_i))}\to 1,~{\rm as}~p_i\to\infty.$$
It follows
\begin{align}
 \frac{R(\phi_{R^{-1}(p_i)t}(x_i))}{R(\phi_{R^{-1}(p_i)t}(y_i))}\le \frac{C_2}{C_1}. \notag
 \end{align}
Hence we obtain
\begin{align}  \frac{R_{N}(x,t)}{R_{N}(y,t)} &=\lim_{i\to\infty}   \frac{   R^{-1}(p_i) R(x_i, R^{-1}(p_i)t)) }{   R^{-1}(p_i) R( y_i,  R^{-1}(p_i)t )}\notag\\
&= \lim_{i\to\infty}   \frac{   R^{-1}(p_i) R(\phi_{R^{-1}(p_i)t}(x_i))}{   R^{-1}(p_i) R(\phi_{R^{-1}(p_i)t}(y_i))}\le \frac{C_2}{C_1}.\notag
 \end{align}
 This proves  (\ref{curvature comparison}).

\end{proof}

The proof of Theorem  \ref{main-theorem} is completed by the following lemma.

\begin{lem}
$(N,g_{N}(t))$ in Lemma \ref{lem-for the main theorem-1} is a shrinking spheres  flow.  Namely,
$$(N,g_{N}(t))=(\mathbb{S}^2,(2-2t)g_{\mathbb{S}^{2}}).$$
\end{lem}
\begin{proof}

By Lemma \ref{lem-for the main theorem-2}, the Gauss curvature of $(N,g_{N}(0))$ has a uniform positive lower bound. Then  $N$ is compact by Myer's Theorem. On the other  hand,  by a classification theorem  of Daskalopoulos-Hamilton-Sesum \cite{DHS},  an  ancient solution on a compact surface  $N$ is either  a shrinking spheres  flow or a  Rosenau solution. The Rosenau solution is obtained by compactifying $(\mathbb{R}\times \mathbb{S}^1(2),h(x,\theta,t)=u(x,t)(dx^2+d\theta^2))$ by adding two points, where $u(x,t)=\frac{\sinh (-t)}{\cosh(x)+\cosh(t)}$ and $t\in (-\infty,0)$. By  a direct computation,
\begin{align}
R_{h(t)}=\frac{\cosh(t)\cosh(x)+1}{\sinh(-t)(\cosh(x)+\cosh(t))}.
\end{align}
It is easy to check that $R_{h(t)}$ doesn't satisfy (\ref{scalar curvature comparison}) in Lemma \ref{lem-for the main theorem-2} as $t\to-\infty$. Hence,  $(N,g_{N}(t))$   must  be a shrinking  spheres   flow on $\mathbb S^2$.  Note that $\widetilde{R}(p_{\infty},0)=1$.  Then it is easy to see that  $g_{\mathbb{S}^{2}}(t)=(2-2t)g_{\mathbb{S}^{2}}$.

\end{proof}

\end{document}